\documentclass[11pt]{amsart}
\usepackage{epsf,latexsym,graphicx}
\usepackage{amssymb}
\usepackage{enumerate}

\newtheorem{theorem}{Theorem}[section]
\newtheorem{lemma}[theorem]{Lemma}
\newtheorem{corollary}[theorem]{Corollary}

\newtheorem{proposition}[theorem]{Proposition}
\theoremstyle{definition}
\newtheorem{definition}[theorem]{Definition}
\newtheorem{example}[theorem]{Example}
\theoremstyle{remark}
\newtheorem{rem}[theorem]{Remark}
\newtheorem*{theorem*}{Theorem}
\newtheorem*{corollary*}{Corollary}
\newtheorem*{ack}{Acknowledgements}

\newcommand\cal[1]{\mathcal {#1}}
\newcommand\pn[1]{{\mathbb P}^{#1}}

\newcommand\zed{{\mathbb Z}}
\newcommand\nat{{\mathbb N}}
\newcommand\rat{{\mathbb Q}}
\newcommand\real{{\mathbb R}}
\newcommand\comp{{\mathbb C}}

\DeclareMathOperator{\codeg}{codeg}
\DeclareMathOperator{\Cayley}{Cayley}
\DeclareMathOperator{\Conv}{Conv}

\title[Classifying smooth lattice polytopes]{Classifying smooth lattice
polytopes \\
via toric fibrations}
\author[A. Dickenstein, S. Di Rocco, R. Piene]{Alicia Dickenstein, Sandra Di
Rocco, Ragni Piene}
\thanks{AD was 
partially supported by UBACYT X042 and X064, CONICET PIP 5617 and ANPCyT PICT
20569, Argentina.}
\thanks{SDR was partially supported by  Vetenskapsr{\aa}det's grant
NT:2006-3539, and the CIAM (Center of industrial and applied mathematics).}
\address{Departamento de Matem\'atica, FCEN, Universidad de Buenos Aires, Ciudad
Universitaria - Pab. I, (1428) Buenos Aires, Argentina}
\email{alidick@dm.uba.ar}
\address{Department of Mathematics, KTH, SE-10044 Stockholm, Sweden}
\email{dirocco@math.kth.se}
\address{CMA/Department of Mathematics, University of Oslo, P. O. Box 1053
Blindern, NO-0316 Oslo, Norway}
\email{ragnip@math.uio.no}
\date{\today }

\begin{document}

\begin{abstract} We show that any smooth $\rat$-normal lattice polytope $P$ of
dimension $n$ and degree $d$ is a strict Cayley polytope if $n \geq 2d+1$. This
gives  a sharp answer, for this class of polytopes,  to a question raised by
V.~V. Batyrev and B. Nill. 
\end{abstract}

%\begin{keyword}
%{lattice polytope, toric variety, toric fibration, Cayley polytope, nef value}
%\end{keyword}

\maketitle

\section{Introduction} \label{sec:intro}

Let $P$ be an $n$-dimensional \emph{lattice polytope} (i.e., a convex polytope with integer vertices) in $\real^n$. We represent it as
an intersection of half spaces 
\begin{equation*} P \, = \, \cap_{i=1}^r H_{\rho_i, -a_i}^+,
\end{equation*}  
where
$H_{\rho_i, -a_i}^+ = \{ x \in \real^n \,|\, \langle \rho_i, x \rangle \geq
-a_i\}$ are the half spaces,
 $H_{\rho_i, -a_i} = \{ x \in \real^n \,|\, \langle \rho_i, x \rangle =
-a_i\}$  the supporting hyperplanes, $\rho_i$  the corresponding primitive
inner normals, and $a_i \in \zed,  \, i =1, \dots, r$. Recall that an
$n$-dimensional lattice polytope  $P$ is {\em smooth } if every vertex is equal
to the intersection of $n$ of the hyperplanes $H_{\rho_i, -a_i}$, and if the
corresponding $n$ normal vectors $\rho_i$ form a lattice basis for
$\zed^n\subset \real^n$. Smooth polytopes are sometimes called Delzant polytopes
or regular  polytopes.

\begin{definition} Let $P=\cap_{i=1}^r H_{\rho_i, -a_i}^+\subset\real^n$ be an
$n$-dimensional lattice polytope. Define $P^{(s)}:=\cap_{i=1}^r H_{\rho_i, -a_i+s}^+,$
for $s\geq 1$.
\end{definition}

Note that the lattice points of 
$P^{(1)}$ are precisely the interior lattice points of $P$. 

\begin{definition} Let $P$ be a smooth  $n$-dimensional lattice polytope and
$s\geq 1$ an integer.  Let $m$ be a vertex of $P$. Reorder the hyperplanes so
that $\{m\}=\cap_{i=1}^n  H_{\rho_i, -a_i}$. We say that $P$ is  {\em $s$-spanned at
$m$} if  the lattice point $m(s)$, defined by $\{m(s)\}=\cap_{i=1}^n  H_{\rho_i, -a_i+s}$, lies in
$P^{(s)}$. We say that $P$ is \emph{$s$-spanned} if $P$ is  $s$-spanned at every
vertex.
\end{definition}

If $\{m\}=\cap_1^n  H_{\rho_i, -a_i}$, we can write $m=(-a_1,\ldots,-a_n)$ in the 
dual basis of $\rho_1,\ldots,\rho_n$, and similarly $m(s)=(-a_1+s,\ldots,-a_n+s)$.
It follows from the definition that if $P$ is  $s$-spanned, then $P^{(s)}\cap
\zed^n \neq \emptyset$.

\begin{example}\label{blowup} Let $P$ be the smooth polytope obtained from the simplex
$d\Delta_3$ by removing the simplex $\Delta_3=\Conv\{(0:0:0),
(1:0:0),(0:1:0),(0:0:1)\}$, see Figure~1. Assume $d\ge 4$. Then
$P^{(1)}\ne\emptyset$, but $P$ is not  $1$-spanned. In fact, consider the vertex $m$ of $P$, given by
$\{m\}=\{(1:0:0)\}=(y=0)\cap(z=0)\cap(x+y+z=1)$, then the lattice point $m(1)$, given by
$\{m(1)\}=\{(0:1:1)\}=(y=1)\cap(z=1)\cap(x+y+z=2)$, is not a point in $P^{(1)}$.
Similarly for the vertices $(0:1:0)$ and $(0:0:1)$.

Note that if we instead remove $2\Delta_3$, then the resulting polytope is 
$1$-spanned. In this case, the vertices $(2:0:0)$, $(0:2:0)$, and $(0:0:2)$ all
``go'' to the same lattice point 
$(1:1:1)$, which is an interior point of the polytope.

\begin{figure}\label{fig:1}
\begin{center}
 {}{}\scalebox{0.40}{\includegraphics{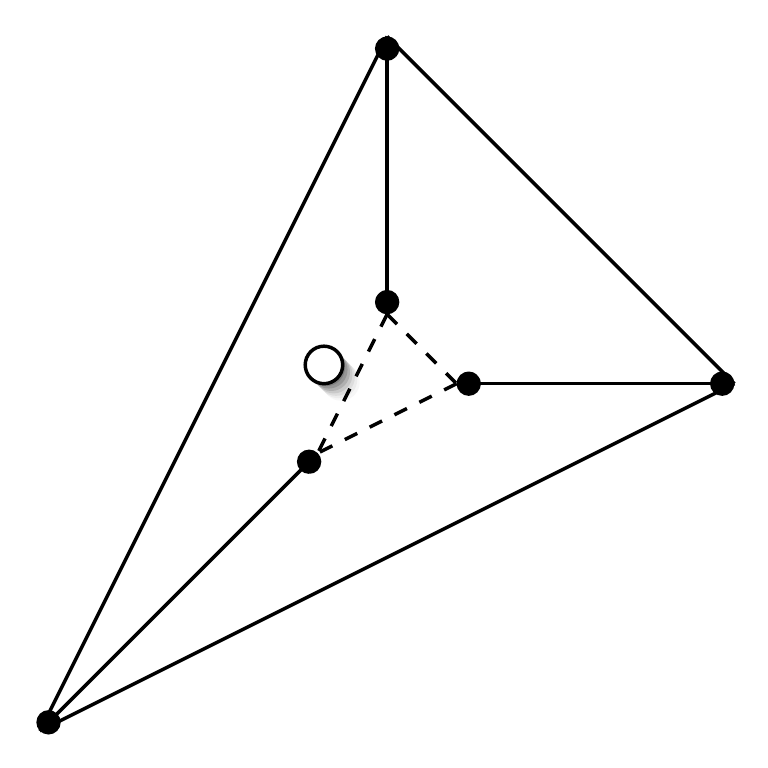}} 
\caption{}\label{Blowup}
 \end{center}
 \end{figure}
\end{example}

We recall the definitions of degree and codegree of a lattice polytope
introduced in   \cite{Ba-Ni}.

\begin{definition}\label{codegree} Let $P\subset\real^n$ be an $n$-dimensional
lattice polytope.   The {\em codegree} of $P$  is the natural number $$
\textstyle\codeg(P):=\min_{\nat}\{ k\, |\, (kP)^{(1)} \cap \zed^n
\neq\emptyset\}.$$
The {\em degree} of $P$ is $$
\deg(P):=n+1-\codeg(P).$$
\end{definition}

We further introduce a more  ``refined" notion of codegree.

\begin{definition}\label{qcodegree} Let $P\subset\real^n$ be an $n$-dimensional
lattice polytope. The  $\rat$-\emph{codegree} of $P$ is defined as $$
\textstyle\codeg_\rat(P):=\inf_{\rat}\{ \frac{a}{b}\, |\,
(aP)^{(b)}\neq\emptyset\}.$$
\end{definition}

The number $\codeg_\rat(P)$ is well defined, since $m(aP^{(b)})=(maP)^{(mb)}$
and, for any polytope $P'$, we have  $P'\neq\emptyset$ if and only if
$mP'\neq\emptyset$ for every integer $m\geq 1$. Moreover, it is clear that
${\codeg_\rat(P)}\leq\codeg(P)$ holds.

\begin{example}\label{simplex} Let $\Delta_n$ denote the $n$-dimensional
simplex. Then we have $\codeg_{\rat}(2\Delta_n)=\frac{n+1}{2}$ and
$\codeg(2\Delta_n)=\lceil{\frac{n+1}{2}}\rceil$.
\end{example}

As we shall see in Proposition~\ref{nefvalue}, the following  definition
embodies the {\em polytope version} of the notion of {\em nef value} for
projective varieties.

\begin{definition} Let $P\subset\real^n$ be an $n$-dimensional smooth lattice
polytope. The {\it nef value} of $P$ is 
$$
\textstyle \tau(P)=\inf_\rat\{\frac{a}{b} \, |\, aP\text{  is  $b$-spanned}\}.$$
\end{definition}

\begin{rem}\label{relation} Clearly, the inequality
 $\tau(P)\geq\codeg_\rat(P)$ always holds. It can be strict, as in Example
\ref{blowup}, where $\codeg(P)=\codeg_\rat(P)=1$ and $\tau(P)=2$.

 Observe also that when $\tau(P)$ is an integer, it follows from Lemma \ref{lem:key} below that $\tau(P)\ge \codeg(P)$.
\end{rem}

\begin{definition} \label{Q-normal} An $n$-dimensional lattice polytope $P$ in
$\real^n$ is called \emph{$\rat$-normal} if  $\codeg_\rat(P)=\tau(P)$.
\end{definition}

We shall now explain the notion of generalized Cayley polytopes --- these are
particular examples of the twisted Cayley polytopes defined in \cite{CaDR}. 
 
\begin{definition}\label{cayley} Let $P_0,\ldots,P_k\subset\real^m$ be lattice
polytopes in $\zed^m$, $e_1,\ldots,e_k$  a basis for $\zed^k$ and
$e_0=0\in\zed^k$. If $\{m_i^j\}_i$ are the vertices of $P_j$, so that $P_j ={\rm
Conv}\{m_i^j\}$ is the convex hull,  and $s$ is a positive integer, consider the
polytope 
${\rm Conv}\{(m_i^j, se_j)\}_{i,j}\subseteq \real^{m+k}$. Any polytope $P$ which
is affinely equivalent to this polytope will be called an
\emph{$s$th order generalized Cayley polytope} associated to $P_0,\ldots,P_k$,
and it will be denoted by 
$$
P\cong [P_0\ast P_1\ast\cdots\ast P_k]^s .$$
If all the polytopes $P_0, \dots, P_k$ have the same normal fan $\Sigma$
(equivalently, if they are strictly combinatorially equivalent), we call $P$
\emph{strict}, and denote it by $$
P\cong\Cayley^s_{\Sigma}(P_0,\ldots,P_k).$$
 If in addition $s=1$, we write $P\cong\Cayley_{\Sigma}(P_0,\ldots,P_k)$ and
call $P$ a \emph{strict Cayley polytope}. 
\end{definition}

Smooth generalized strict Cayley polytopes are natural examples of 
$\rat$-normal polytopes. In
Proposition~\ref{localsplit} we give  sufficient conditions for  $P\cong
\Cayley^s_{\Sigma}(P_0,\ldots,P_k)$ to be
  $\rat$-normal, and we compute the common value $\codeg_\rat(P) = \tau(P)$ in
this case.
  \medskip

%%%%%%%%%%%%%%%%%%%%%%%%%%%%%%
In \cite{Ba-Ni} Batyrev and Nill   posed the following question:
\medskip

\noindent {\bf {\sc Question}}. {\em Given an integer $d$,  does there exist an
integer $N(d)$ such that every lattice polytope of degree $d$ and dimension
$\geq N(d)$ is a Cayley polytope?}
\smallskip

\noindent 
A first general answer was given by C. Haase, B. Nill, and S. Payne
in~\cite{hnp}, where they prove the existence of a lower bound which is quadratic
 in $d$.
\medskip

In this article we give an optimal linear bound in the case of smooth
$\rat$-normal lattice polytopes,  and show that polytopes of dimension greater
than or equal to this bound are  strict Cayley polytopes.
\smallskip

\noindent {\bf {\sc Answer}}. For $n$-dimensional smooth $\rat$-normal lattice
polytope, we can take $N(d)=2d+1$. More precisely, if $P$ is a smooth
$\rat$-normal lattice polytope of dimension $n$ and degree $d$ such that $n\geq 
2d+1$, then $P$ is a strict Cayley polytope.
\medskip

Observe that the condition $n\geq  2\deg (P)+1$ is equivalent to $\codeg (P)\ge
\frac{n+3}2$. Furthermore, note that our bound is sharp: consider the standard
$n$-dimensional simplex $\Delta_n$ and let $P:= 2 \Delta_n$ as in
Example~\ref{simplex}. 
Then, $\tau(P) = \codeg_\rat(P) = \frac{n+1} 2 = \frac{n+3} 2 -1$ and $P$ is not
 a  Cayley polytope. It is, however, a generalized Cayley polytope, with $s=2$.
We conjecture that an even smaller  linear bound, like $\frac{n+1}{s}$, should
suffice for the polytope to be an $s$-th  order generalized Cayley polytope. 

We shall deduce our answer from Theorem~\ref{defective} below, which gives a
characterization of $\rat$-normal smooth lattice polytopes with big codegree.
Before stating our main theorem, we recall the notion of a defective projective
variety.

\begin{definition} \label{def:defective} Let $X \subset \pn{N}$ be a projective
variety over an algebraically closed field and denote by $X^* \subset
(\pn{N})^\vee$ its dual variety in the dual projective space. Then $X$ is
defective if  $X^*$ is not a hypersurface, and its defect is the natural number
$\delta  = N-1 - \dim(X^*)$. 
\end{definition}

\begin{theorem}\label{defective} Let $P\subset \real^n$ be a smooth
 lattice polytope  of dimension $n$. Then the following statements
are equivalent:
\begin{enumerate}
\item[{\rm (1)}] $P$ is $\rat$-normal and $\codeg(P)\geq \frac{n+3}{2}$,   
\item[{\rm (2)}] $P=\Cayley_{\Sigma}(P_0,\ldots,P_k)$ is a smooth strict Cayley
polytope, where $k+1=\codeg(P)$ and $k>\frac{n}{2}$,
\item[{\rm (3)}]  the (complex) toric polarized variety $(X,L)$ corresponding to
$P$ is defective, with defect $\delta = 2 \codeg (P) -  2 - n$.
\end{enumerate}
\end{theorem}
\medskip

We conjecture that the assumption $\tau(P)=\codeg_\rat(P)$ always holds for
smooth lattice polytopes satisfying $\codeg(P)\geq \frac{n+3}{2},$ and we
therefore expect that the above classification holds for all smooth  polytopes.
\medskip

Our proof of Theorem  \ref{defective} relies on the study of the nef value of
nonsingular toric polarized varieties.  This is developed in Section~\ref{sec:nef}. 
Section~\ref{sec:cayley} contains the study of generalized Cayley polytopes in
terms of fibrations. In particular, for smooth generalized strict Cayley
polytopes $P^s =\Cayley_{\Sigma}^s(P_0,\ldots,P_k)$ such that $ s < k+1$  and 
$\dim P_i + 1 < \frac{k+1}s$ for all $i$, 
Proposition~\ref{localsplit} shows that $\codeg(P^s) > 1$,  providing a family of
lattice polytopes without interior lattice points. Section~\ref{sec:class}
contains the proof of the classification Theorem~\ref{defective}, together with
some final comments.
%%%%%%%%%%%%%%%%%%%%%%%%%%%%%%%%%%%%%%%%%%%%

\section{The codegree and the nef value} \label{sec:nef}

Let $X$ be a nonsingular projective variety over an algebraically closed field,
and let $L$ be an ample line bundle (or divisor) on $X$. 

\begin{definition} Assume that the canonical divisor $K_X$ is not nef. The {\it
nef value} of $(X,L)$ is defined as $$
\textstyle \tau_L:=\min_\real\{t\,|\, K_X+tL \text{ is nef}\}.$$
\end{definition}
\noindent By Kawamata's rationality theorem \cite[Prop. 3.1, p.~619]{kaw}, the nef value $\tau_L$
is a positive rational number.
 
\begin{proposition}\label{nefvalue} Let $X$ be a nonsingular projective toric
variety of dimension $n$, and let $L$ be an ample line bundle on $X$.  Let
$P\subset\real^n$ be the associated $n$-dimensional smooth lattice polytope.
Then   
 $$
\tau_L=\tau(P).$$
\end{proposition} 

\begin{proof}  If we write the
polytope $P$ as $\cap_{i=1}^r H^+_{\rho_i,-a_i}$, then $L=\sum_{i=1}^r a_iD_i$, where
the $D_i$ are the invariant divisors on $X$. Since $K_X=-\sum_{i=1}^r D_i$,
the polytope associated to $X$ and the
adjoint line bundle $bK_X+aL$ is 
$$\textstyle P_{bK_X+aL}=\cap_{i=1}^r H^+_{\rho_i,-a\cdot
a_i+b}=(aP)^{(b)}.$$
The lattice points of $(aP)^{(b)}$ form a basis for the
vector space of global sections of $bK_X+aL$, 
$$
H^0(X,bK_X+aL)=\oplus_{m\in (aP)^{(b)}\cap M}\comp \chi^m$$
(see \cite[Lemma 2.3, p.~72]{oda}).

Denote by $\Sigma$ the fan of $X$, 
let $x(\sigma)$ be the fixed point associated to the $n$-dimensional cone
$\sigma=\langle\rho_1,\ldots,\rho_n\rangle\in\Sigma$, and let
$U_\sigma=X\setminus(\cup_{\tau\not\subset\sigma} V(\tau))$ be the affine patch
containing $x(\sigma)$. The restriction of a generator 
$\chi^m\in H^0(X,bK_X+aL)$ to $U_\sigma$ is $$
\chi^m|_{U_\sigma}=\Pi_1^n \chi_i^{\langle m,\rho_i\rangle-(-a\cdot a_i+b)},$$
where $\chi_1,\ldots,\chi_n$ is a system of local coordinates such that
$x(\sigma)=(0,\ldots,0).$ It follows that the line bundle $bK_X+aL$ is spanned,
or  globally generated, at $x(\sigma)$ (i.e., it has at least one nonvanishing
section  at $x(\sigma)$) if and only if the lattice point $(-a\cdot
a_1+b,\ldots,-a\cdot a_n+b)$, written with respect to the dual basis of $\rho_1,\ldots, \rho_n$, is in $(aP)^{(b)}$.

Because the base locus of a line bundle is invariant under the torus action, if
it is non empty, it must be the union of invariant subspaces. Hence it has to
contain fixed points. It follows that $bK_X+aL$ is spanned if and only if it is
spanned at each fixed point, hence if and only if the polytope $aP$ is
 $b$-spanned.
\end{proof}

\begin{corollary}\label{n+1} 
Let $P\subset\real^n$ be an $n$-dimensional smooth lattice polytope. Then
\begin{enumerate}
\item[{\rm (1)}] $\codeg(P)=\codeg_\rat(P)=\tau(P)=n+1$ if and only if
$P=\Delta_n.$
\item[{\rm (2)}] If $P\neq \Delta_n,$ then $\codeg_\rat(P)\leq \tau(P)\leq n.$ 
\end{enumerate}
\end{corollary}

\begin{proof} 
Let $(X,L)$ be the polarized nonsingular toric variety corresponding to $P$. In
\cite[Cor. 4.2]{Mustata} it is proven that $K_X+H$ is spanned for any line
bundle $H$ on $X$ such that $H\cdot C\geq n$ for all invariant curves $C$ on
$X$, unless $X=\pn{n}$ and $H={\cal O}_{\pn{n}}(n).$ 

Because the line bundle $L$ is ample, we have $L\cdot C\geq 1$ and thus $nL\cdot
C\geq n$, for all invariant curves $C$.  It follows that if $X\neq \pn{n}$ or
$nL\neq {\cal O}_{\pn{n}}(n)$, then $\tau_L=\tau(P)\le n$. This proves $(2)$. If
$(X,L)=(\pn{n},{\cal O}_{\pn{n}}(1))$, then $\tau_L=\tau(P)=n+1$, because
$K_{\pn{n}}={\cal O}_{\pn{n}}(-n-1)$. The corresponding polytope is $P=\Delta_n$
and $\codeg(P)=\codeg_\rat(P)=\tau(P)=n+1$, as stated in $(1)$. 
\end{proof}

\begin{lemma} \label{lem:key} Let $P\subset \real^n$ be a smooth $n$-dimensional
polytope with codegree $c$ and nef value $\tau$. Then $\tau \in \rat_{>0}$ and 
$\tau > c-1$.
\end{lemma}

\begin{proof} Let $(X,L)$ be the polarized projective toric variety associated
to $P$. Then $\tau= \tau_L \in {\rat}_{>0}$. By \cite[Lemma~0.8.3]{ BESO92}
$K_X+\tau L$ is nef (and not ample). For any $s \geq \tau$, $K_X + s L =
(K_X+ \tau L) + (s-\tau)L$ is also nef. When $s$ is an integer, this implies
that $K_X + sL$ is spanned, hence $(sP)^{(1)}
\cap \zed^n \neq \emptyset$. Taking $s=c-1$ and observing that $((c-1)P)^{(1)}$
has no lattice points, we deduce that $c-1 < \tau$, as claimed.
\end{proof}

Let $\tau(P)=\frac{a}{b},$ where $a,b$ are coprime. On complete toric varieties,
nef line bundles with integer coefficients are spanned
(see~\cite[Thm.~3.1]{Mustata}). It follows that 
 the linear system $|bK_X+a L|$ defines a morphism 
$$
\varphi: X\to \pn{N},$$
where $N=|(aP)^{(b)}\cap \zed^n|-1$. The Remmert--Stein factorization gives
$\varphi=f\circ\varphi_P,$ where $\varphi_P:X\to Y$ is a  morphism with
connected fibers onto a normal toric 
variety $Y$ such that $\dim Y=\dim (aP)^{(b)}$ and $f:Y \to \pn{N}$.  Moreover,
$\varphi_P$ is the contraction of a face of the nef cone NE($X$) \cite[Lemma
4.2.13, p.~94]{BESO}.

\begin{rem}\label{integer} If the morphism $\varphi_P$ contracts a line, i.e.,
if there is a curve $C$ such that $L\cdot C=1$ and $\varphi_P(C)$ is a point,
then $\tau(P)$ is  necessarily an integer. In fact $0=(bK_X+aL)\cdot C$ implies
$\frac{a}{b}=-K_X\cdot C\in \zed$.
\end{rem}

\begin{lemma}\label{line} Let $P$ be a smooth $n$-dimensional polytope and let
$(X,L)$ be the corresponding polarized toric variety. If $\tau(P)>\frac{n+1}2$,
then there exists an invariant line on $X$ contracted by the nef value morphism
$\varphi_{P}$. In particular, 
$\tau(P)\in\zed$, and $\tau(P)\ge \codeg(P)$. If, moreover, $\varphi_{P}$ is not
birational, then $\varphi_{P}$ is the contraction of an extremal ray in the nef
cone, unless $n$ is even and $P=\Delta_{\frac{n}{2}}\times
\Delta_{\frac{n}{2}}$.
\end{lemma}

\begin{proof} 
Because the nef value morphism is the contraction of a face of the Mori cone, it
contracts at least one extremal ray.  Take $C$ to be a generator of this ray. 
Recall that, by Mori's Cone theorem (see e.g. \cite[p.~25]{CKM88}), $n+1\ge
-K_{X}\cdot C$. Because  $(K_{X}+\tau(P) L)\cdot C=0$, we have 
\[
\textstyle n+1\ge -K_{X}\cdot C=\tau(P) L\cdot C > \frac{n+1}2 L\cdot C,
\]
which gives $L\cdot C=1$ and $\tau(P) = -K_{X}\cdot C \in \zed$. 
Lemma~\ref{lem:key}  gives  $\tau(P) > \codeg(P) -1$ from which we deduce that
$\tau(P) \ge \codeg(P)$. 
If $\varphi_{P}$ is not birational, the last assertion follows from 
\cite[(3.1.1.1), p.~30]{BSW}.
\end{proof}

We will also need the following key lemma. This lemma, and its proof, is essentially the same as \cite[Lemma 7.1.6, p.~157]{BESO}.

\begin{lemma}\label{nefmap} Let $(X,L)$ be the polarized nonsingular toric
variety associated to a smooth $\rat$-normal lattice polytope $P$. Then the
morphism $\varphi_P$ is not birational.
\end{lemma}

\begin{proof} Let $\frac{a}{b}=\tau(P).$ Assume the morphism $\varphi_P$ is 
birational. By \cite[Lemma 2.5.5, p.~60]{BESO} there is an integer $m$ and an
effective divisor $D$ on $X$ such that $$
m(bK_X+a L)=L+D.$$
It follows that $D\in |mbK_X+(ma-1)L|$, and thus $\codeg_\rat(P)
\leq\frac{ma-1}{mb}<\frac{a}{b} = \tau(P)$, which contradicts the assumption
that $P$ is $\rat$-normal.
\end{proof}
%%%%%%%%%%%%%%%%%%%%%%%%%%%

\section{Generalized Cayley Polytopes and toric fibrations}\label{sec:cayley}

Strict generalized Cayley polytopes (recall Definition~\ref{cayley})
  correspond to particularly nice toric fibrations, namely projective bundles.
We will  compute their associated nef values and codegrees. We refer to
\cite[Section 3]{CaDR} for further details on toric fibrations.

\begin{definition}\label{toricfibration} A \emph{polarized toric  fibration} is
a quintuple $(f, X, Y, F,L),$ where
\begin{enumerate}
\item $X$ and $Y$ are normal toric varieties with $\dim(Y)<\dim(X)$,
\item $f\colon X\to Y$ is an equivariant flat surjective morphism  with
connected fibers,
\item the general fiber of $f$  is isomorphic to the (necessarily toric) variety
$F$,
\item $L$ is an ample line bundle on $X$.
\end{enumerate}
\end{definition}

There is a $1$-$1$ correspondence between polarized toric fibrations and
fibrations of polytopes, making Definition \ref{toricfibration} equivalent to
the following. 

\begin{definition} Let $\pi:M\to \Lambda$ be a surjective map of lattices and
let $P_0,\ldots,P_k\subset M_\real$ be lattice polytopes. We call $\pi$ a
\emph{fibration with fiber $\Delta$} if
\begin{enumerate}
\item $\pi_\real(P_i)=m_i\in\Lambda$ for every $i=0,\ldots,k$,
\item $m_0,\ldots,m_k$ are all distinct and are the vertices of $$
\Delta:=\Conv\{m_0,\ldots,m_k\}\subset \Lambda_{\real},$$
\item $P_0,\ldots,P_k$ have the same normal fan, $\Sigma$.
\end{enumerate}
\end{definition}

In \cite[Lemma 3.6]{CaDR} it is proven that $(f, X, Y, F,L)$ is a toric
fibration if and only if the polytope $P\subset M_\real$ associated to $(X,L)$  
has the structure of a fibration. More precisely, $(f, X, Y, F,L)$ is a toric
fibration if and only if there is a sublattice $\Lambda^{\vee}\hookrightarrow
M^{\vee}$ such that the dual map $\pi:M\to \Lambda$ is a fibration of polytopes
with fiber $\Conv\{\pi(P)\}\subset \Lambda\otimes_\zed \real$. Moreover, $F$ is
the toric variety defined by the inner normal fan of $\Conv\{\pi(P)\}$, and
every fiber of $f\colon X\to Y$, with the reduced scheme structure, is
isomorphic to $F$.

 Observe that by construction, $\Cayley^s_{ \Sigma}(P_0,\ldots,P_k)\subset
M_\real\cong
\real^m \oplus \Lambda_\real$, where  $\Lambda_\real=\real^{k}$. The projection
$$
\pi :M\to \Lambda,  \quad \quad \pi(m,e)=e,$$
gives the polytope $\Cayley^s_{ \Sigma}(P_0,\ldots,P_k)$ the structure of a  a
fibration with fiber $s\Delta_k$.

It follows that $\Cayley^s_{ \Sigma}(P_0,\ldots,P_k)$ defines a toric fibration
$f\colon X\to Y$ with general fiber isomorphic to $\pn{k}$. 

\begin{example} The strict Cayley sum
$\Cayley_{\Sigma(\Delta_1)}^2(4\Delta_1,2\Delta_1,2\Delta_1)$
 is associated to  
the toric fibration $(\pi, \pn{}_{\pn{1}}(\mathcal O_{\pn{1}}(4)\oplus\mathcal
O_{\pn{1}}(2)\oplus\mathcal O_{\pn{1}}(2)), \pn{1}, \pn{2}, \xi)$, where $\xi$
is the tautological line bundle and $\pi$ is  the projection (see Figure~2).
\begin{figure}\label{fig:2}
 \begin{center}
 {}{}\scalebox{0.75}{\includegraphics{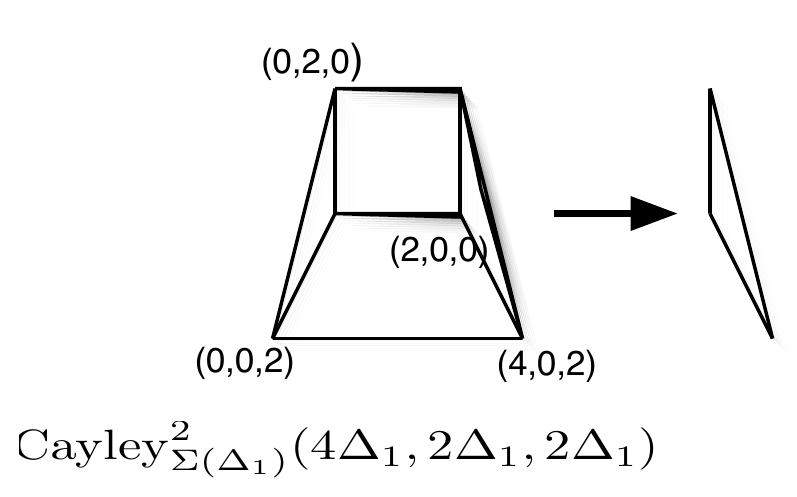}}
\caption{}\label{Cayley}
 \end{center}
 \end{figure}
 \end{example}

\begin{rem} Even if all the $P_i$ are smooth and the fiber $\Delta$ is smooth,
the polytope $\Cayley^s_{ \Sigma}(P_0,\ldots,P_k)$ is not necessarily smooth.
Consider the polytope $\Cayley_{\Sigma(\Delta_1)}^2(6\Delta_1,
5\Delta_1,3\Delta_1)$ depicted in Figure~3. 
At the vertex $(3,0,2)$, the first lattice points on the corresponding three
edges give the vectors $(2,0,2)-(3,0,2)=(-1,0,0)$, $(4,1,1)-(3,0,2)=(1,1,-1)$,
and $(6,0,0)-(3,0,2)=(3,0,-2)$, which do not form a basis for the lattice.
\begin{figure}\label{fig:3}
\begin{center}
 {}{}\scalebox{0.60}{\includegraphics{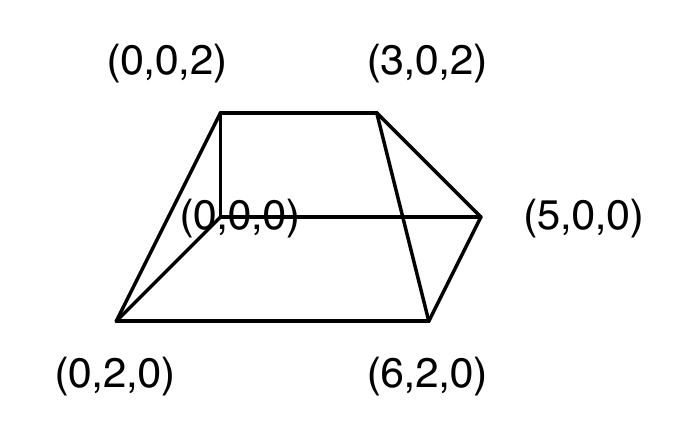}}
 \caption{}\label{Singular}
 \end{center}
 \end{figure}
 \medskip
 \end{rem}
 
 \begin{rem}
 When the polytopes $P_i$ are not strictly combinatorially equivalent, the
variety associated to the Cayley polytope $[P_0\star\cdots \star P_k]^s$ is birationally equivalent to a toric variety associated to a strict Cayley polytope, in the following precise way.
 
 The normal fan $\Sigma$  defined by the Minkowski sum $P_0+\cdots + P_k$  is a
common refinement of the normal fans defined by the polytopes $P_i$ (see e.g.
\cite[7.12]{Z95}). Let $(X_{P_i},L_i)$ be the polarized toric variety associated
to the polytope $P_i$ and let $\pi_i:Y\to X_{P_i}$ be the induced birational
morphism, where $Y$ is the toric variety defined by the fan $\Sigma$.
 Notice that the line bundle $\pi_i^*L_i$ on $Y$ is spanned, and the associated
polytope $Q_i$ is
 affinely equivalent to $P_i$.
 
 Set $P'_i=Q_i+\sum_{j=0}^k P_j$ for $i=0,\ldots,k$. Note that the $P'_i$ are strictly combinatorially equivalent, since their inner normal fan is $\Sigma$.
 The normal fan of the polytope $\Cayley^s_{\Sigma}(P'_0,\ldots,P'_k)$ is then a
refinement of the normal fan of the polytope $[P_0\star\cdots \star P_k]^s$, and
thus it defines  a  proper birational morphism 
 $$
\pi: X_{\Cayley^s_{\Sigma}(P'_0,\ldots,P'_k)}\to X_{[P_0\star\cdots \star
P_k]^s}.$$
 \end{rem}
 
\begin{lemma}\label{smooth} Let  $P_0,\ldots,P_k\subset\real^m$ be  strictly
combinatorially equivalent polytopes such that  the polytope
Ê$P^s=\Cayley^s_{\Sigma}(P_0,\ldots,P_k)$ is smooth. Let $f\colon X\to Y$ be the
associated toric fibration. Then $P_i$ is smooth for all $i=0,\ldots,k$, and 
all fibers are reduced and  isomorphic to $\pn{k}$.
\end{lemma}

\begin{proof} Let $F$ be an invariant  fiber of $f$, then, since it is not
general, it must be the fiber over a fixed point of $Y$. Equivalently, there is
a vertex $m$ of $\Cayley_{\Sigma}^s(P_0,\ldots,P_k)$ which is the intersection
of a $k$-dimensional face $Q$ such that $\pi(Q)=s\Delta_k$ and an
$(n-k)$-dimensional face $R$ such that $\pi(R)=e_i\in \real^k$. Hence $R=P_i$. 
Moreover, by construction, $Q$ is a simplex (possibly non standard) of edge
lengths $b_1,\ldots,b_k$, with $1\leq b_i\leq s$. Because
$\Cayley^s_\Sigma(P_0,\ldots,P_k)$ is smooth, the first lattice points 
$m_1,\ldots,m_n$ on the $n$ edges
meeting at $m$ form a lattice basis.  This is equivalent to asking that the
$n\times n$ matrix formed by taking the integral vectors 
$m_1-m,\ldots,m_n-m$ as columns, has determinant $\pm 1$.
After reordering we can assume that $e_0=0$. Then the corresponding matrix $A$
has the shape featured in Figure~4,
\begin{figure}\label{fig:4}
\begin{center}
 {}{}\scalebox{0.40}{\includegraphics{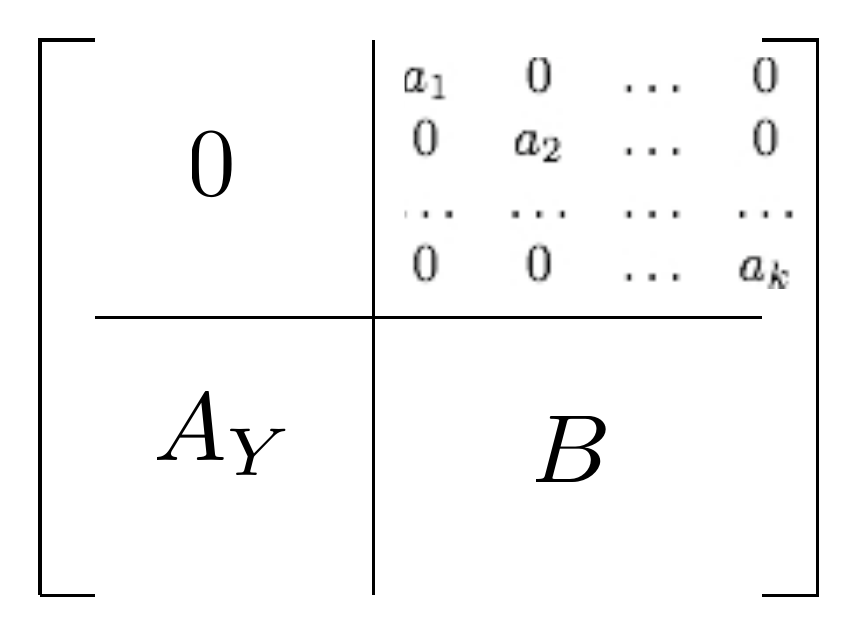}}
\caption{}\label{Matrix}
\end{center}
\end{figure} 
where  $a_i$, $1\leq a_i\leq s$, corresponds to the coordinate of the first
lattice point on the $i$th edge of the simplex $Q$. The matrix $A_Y$ is the
matrix given by the first lattice points through the (corresponding) vertex of
$P_0$. A standard computation in linear algebra gives
$\det(A)=\det(A_Y)a_1\cdots a_k$. It follows that $\det(A_Y)=1$ and $a_1=\cdots
=a_k=1$.

Because each vertex of $P_i$ is the intersection of a smooth fiber with $P_i$,
we conclude that $P_i$ is smooth for all $i=0,\ldots, k$. The equalities
$a_1=\cdots = a_k=1$ show that all invariant fibers have edge length $s$. Every 
special fiber (as a cycle) is a combination of invariant fibers. If there were a
non reduced fiber $tF$, then it should contain an invariant curve such that $tC$
is numerically equivalent to $C'$, where $C'$ is an invariant curve in 
$s\Delta$.  Then, because $L\cdot C\geq s$ and $L\cdot C'=s$, we would have
$t=1$. Lemma \ref{smooth} implies that the $P_i$ are smooth, and hence $Y$ is
smooth. One can see this also from the standard fact that for every morphism
with connected fibers between two normal toric varieties, the general fiber 
is necessarily toric \cite[Lemma 1.2]{DRS04}.
\end{proof}

Observe that the hypothesis that $P^s$ is smooth in Lemma~\ref{smooth} above,
is essential in order to prove that all fibers are reduced and embedded as Veronese varieties.

Toric projective
bundles are isomorphic to projectivized bundles of a vector bundle, which
necessarily splits as a sum of line bundles \cite[Lemma 1.1]{DRS04}. We will
denote the projectivized bundle by $\pn{}_Y(L_0\oplus\cdots \oplus L_k)$, where $Y$ is the
toric variety associated to $\Sigma$.

 \begin{proposition}
 Let  $P_0,\ldots,P_k\subset\real^m$ be  strictly combinatorially equivalent
polytopes such that 
  $$
P^s:=\Cayley_{\Sigma}^s(P_0,\ldots,P_k)$$
is smooth. Then  there are   line bundles $L_0,\cdots,L_k$ on $Y=X(\Sigma)$ such
that the toric variety 
$X(\Sigma_{P^s})$, defined by the inner normal fan $\Sigma_{P^s}$ of 
 $P^s$, is isomorphic to $\pn{}_Y(L_0\oplus \cdots \oplus L_k).$ 
 \end{proposition}
\begin{proof} Denote by $L$ the ample line bundle on $X(\Sigma_{P^s})$
associated to the given polytope $P^s$. By Lemma \ref{smooth}, all
fibers are isomorphic to 
$\pn{k}$ and thus  $X(\Sigma_{P^s})$ has the structure of a projective bundle
over $Y$. Equivalently, $f_*(L)=L_0\oplus \cdots \oplus L_k$ and
$X(\Sigma_{P^s})\cong\pn{}_Y(L_0\oplus \cdots \oplus L_k)$, where  $f\colon
X(\Sigma_{P^s})\to Y$.
\end{proof}

A complete description of the geometry of such varieties when $s=1$ is contained
in \cite[Section 3]{DiR06} and \cite[1.1]{oda}.
\medskip

Throughout the rest of the section we will always assume that
$P^s=\Cayley^s_{\Sigma}(P_0,\ldots,P_k)$ is a smooth polytope. Denote as before
by $(X, L)$ the associated polarized toric variety. Let $f\colon X
\to Y$ be as above. The invariant curves on $X$ are of two types:
\begin{enumerate}
\item pullbacks $f^*V(\alpha_i)$ of invariant curves from $Y$,  corresponding to
the edges of $P_i$,
\item curves $V(\alpha_F)$ contained in a fiber $F$, corresponding to the edges
on simplices $s\Delta^k$.
\end{enumerate} 
Line bundles on $X$ can be written as $L=f^*(M)+a\xi $, where  $M$ is a line
bundle on $Y$ and $\xi$ is the tautological line bundle on $\pn{}_Y(L_0\oplus
\ldots \oplus L_k).$  

Because a line bundle on a nonsingular toric variety is spanned, respectively
(very) ample,  if and only if the intersection with all the invariant curves is
non negative, respectively positive, there is a well understood spannedness and
ampleness criterion, see \cite[Prop.~2]{DiR06}.

\begin{lemma}\label{ample} Let  $L=f^*(M)+a\xi $ be a line bundle on 
$\pn{}_Y(L_0\oplus\cdots
\oplus L_k )$. Then for every curve of type $f^*V(\alpha_i)$, we have $L\cdot
f^*V(\alpha_i)=(M+sL_i)\cdot V(\alpha_i)$, and for every curve $V(\alpha_F)$ we have
$V(\alpha_F)\cdot L=a$. Consequently
\begin{enumerate}
\item[{\rm (1)}]  $L$ is ample if and only if $a\geq 1$ and $M+sL_i$ is ample 
for all $i$,
\item[{\rm (2)}]  $L$ is spanned if and only if $a\geq 0$ and  $M+sL_i$ are
spanned  for all $i$.
 \end{enumerate}
\end{lemma}

Because the fibers of $\pi$ correspond to simplices $s\Delta_k,$ and thus are
embedded as $s$-Veronese varieties, and because the line bundles $L_i$ are ample, we see that
$P^s=\Cayley^s_{\Sigma}(P_0,\ldots,P_k)$ corresponds to the toric embedding $$
(\pn{}_Y(L_0\oplus\cdots \oplus L_k ), s\xi).$$

%%%%%%%revision starts here%%%%%%%%%%%%%
\begin{proposition}\label{localsplit}
 Let $P^s =\Cayley^s_{\Sigma}(P_0,\ldots,P_k)$ be a smooth generalized strict
Cayley polytope. Assume that $\dim P_i + 1 < \frac{k+1}s$ for all $i$. 
 Then $P^s$ is $\rat$-normal, and
 $$
\textstyle\codeg_\rat(P^s) = \tau(P^s)=\frac{k+1}{s}.$$
\end{proposition}

\begin{proof} 
Recall that
$$X:=X(P^s)=\pn{}_Y(L_0\oplus\cdots \oplus  L_k),$$
where $L_i$ is ample on $Y:=X(\Sigma)$, for $i=0,\ldots,k$, and $P^s$ is the 
polytope defined by the line bundle $s\xi$, where $\xi:=\xi_{L_0\oplus\cdots \oplus  L_k}$.

Recall also that the canonical line bundle on $X$ is $K_X=\pi^*(K_Y+L_0+\cdots
+L_k)-(k+1)\xi$, where $\pi\colon X\to Y$. It follows that 
$$H:=
bK_X+as\xi=\pi^*(b(K_Y+L_0+\cdots +L_k))+(as-b(k+1))\xi.$$
By Lemma \ref{ample}, $H$ is spanned (resp. ample) if and only if $as-b(k+1)\geq 0$ (resp. $\geq 1$) and 
$b(K_Y+L_0+\cdots +L_k)+sL_i$ is spanned (resp. ample).

Observe that, because $L_i$ is ample for each $i$, 
$K_Y+L_0+\cdots +L_k$ is ample if $k+1> \dim P_i +1$ \cite[Cor.~4.2 (ii)]{Mustata}, which holds by assumption. It follows that, with the given hypotheses, we have 
\begin{itemize}
\item[(i)]$H$ is spanned if and only if $\frac{a}b\ge \frac{k+1}s$
\item[(ii)] $H$ is ample if and only if $\frac{a}b> \frac{k+1}s$,
\end{itemize}
and thus $\tau(P)=\frac{k+1}s$.

Consider the projection map $\pi\colon \real^n\to \real^k$ such that $\pi(P)=s\Delta_k$. Clearly, if $a,b$ are such that $(aP)^{(b)}\neq \emptyset$, then for all points $m\in (aP)^{(b)}$, $\pi(m)\in (as\Delta_k)^{(b)}$. This implies that $\frac{a}b\ge \codeg_\rat(s\Delta_k)=\frac{k+1}s$, and hence $\codeg_\rat(P)\ge \frac{k+1}s$. Together with the inequality $\codeg_\rat(P)\le \tau(P)=\frac{k+1}s$, this proves
$$\textstyle\codeg_\rat(P)= \tau(P)=\frac{k+1}s.$$
\end{proof}

\begin{rem} \label{rmk:Haase}
Christian Haase showed us the following beautiful geometric argument to prove Proposition~\ref{localsplit}.
As before, denote by  $\Sigma$  the common normal fan of $P_0, \dots, P_k$ and by
$\Sigma(P^s)$ the normal fan of $P^s$. Consider the smooth toric varieties $X = X(\Sigma(P^s))$ and $ Y=X(\Sigma)$, and let
$L$ denote the ample line bundle on $X$ with associated polytope $P^s$ and  $L_0, \dots, L_k$
the ample line bundles on $Y$ with associated polytopes $P_0, \dots, P_k$.

As $\tau(s \Delta_k) = \frac{k+1}{s} > \tau(P_i)$, we have that an integer multiple of
 $L_i + \frac s {k+1} K_Y$ is ample on $Y$ for all $i=0, \ldots,k$.
Let $\rho_1, \ldots, \rho_\ell$ be the primitive
 generators of the one dimensional cones in $\Sigma$, so that
 \[ P_i \, = \, \cap_{j=1}^\ell H^+_{\rho_j, - a_{ij}}. \]
Then the polytope 
\[  \cap_{j=1}^\ell H^+_{\rho_j, - a_{ij}+ \frac{s} {k+1}}, \]
corresponding to the line bundle (with coefficients in $\rat$) $L_i + \frac s {k+1} K_Y$, is combinatorially equivalent
to $P_i$,  for all $i=0, \ldots, k$. On the other hand, the polytope
$$\textstyle\{ x=(x_1,\ldots,x_k) \in s \Delta_k\,|\, \sum_{i=1}^k x_i \le s - \frac s {k+1}, \, x_i  \ge \frac s {k+1}, \, \forall \, i=1, \ldots, k \}$$
 equals the barycenter $v$ of the simplex $s \Delta_k$. 
 Hence the polytope associated to $L + \frac{s} {k+1} K_{X}$ reduces to the fiber over
$v$ in $P^s$,  which can be identified with $\frac 1 {k+1} (P_0 + \cdots +P_k)$, 
and no multiple of the corresponding line bundle is ample. 

For any rational number $\frac a b > \frac {k+1} s$, we have again that  the polytope
\[  \cap_{j=1}^\ell H^+_{\rho_j, - a_{ij}+ \frac{b} a} \]
 is combinatorially equivalent to $P_i$ for all $i=0, \ldots, k$.  But now the polytope given by
 the points in $s \Delta_k$ at lattice distance $\frac b {a}$ from each of its facets is also
combinatorially equivalent to $s \Delta_k$, and therefore the polytope corresponding to
$L + \frac{b} {a} K_{X}$ is combinatorially
equivalent to $P^s$. We conclude that $\tau(P^s) = \frac{k+1} s$, as wanted. 
\end{rem}

\section{Classifying smooth lattice polytopes with high codegree}
\label{sec:class}

We now use the results in the previous sections to give the proof of The\-o\-rem
\ref{defective}.
\medskip

\noindent {\bf Proof of Theorem~\ref{defective}.}  Let $(X,L)$ be the nonsingular toric variety 
and ample line bundle associated to $P$. 

Assume (1) holds. Because $P$ is $\rat$-normal, Lemma \ref{nefmap} implies that the nef value map $\varphi:=\varphi_{P}$ is
not birational. 
Set $\tau:=\tau(P)$.  Lemma~\ref{lem:key}
gives that $\tau > \codeg(P)-1\ge \frac{n+1}2$. If $n$ is even and $P=
\Delta_{\frac{n}{2}}\times  \Delta_{\frac{n}{2}}$, then 
$\codeg(P)=\frac{n}{2}+1< \frac{n+3}2$. Therefore, Lemma \ref{line} 
implies that $\tau$ is an integer and  $\tau \ge \codeg(P)\ge \frac{n+3}2$, and
that  
$\varphi\colon X\to Y$ is a (non birational) contraction of an
extremal ray in the nef cone $NE(X)$ of $X$. 

By \cite[Cor.~2.5, p.~404]{Reid83}, we know that $\varphi$ is flat, $Y$ is a
smooth toric variety, and, since $X$ is smooth, a general fiber $F$ is
isomorphic to $\pn{k}$, where $k=\dim X - \dim Y$. Under this isomorphism, $L|_F
={\mathcal O}_{\pn{k}}(s)$, for some positive integer $s$. Let $\ell\subseteq F\cong
\pn{k}$ be a line. Then, since $F$, and hence $\ell$, is contracted by
$\varphi$, we have 
\[
0=(K_X+\tau L)\cdot \ell=K_X\cdot \ell+\tau L\cdot \ell = K_F\cdot \ell+\tau
s=-(k+1)+\tau s.
\]
We therefore get
\begin{equation*}\label{ineq}
\textstyle \frac{n+1}2< \tau = \frac{k+1}s \le \frac{n+1}{s},
\end{equation*} which gives $s=1$ and $\tau=\codeg_\rat(P) =k+1 \in \zed$. 
By Lemma~\ref{lem:key}, we have $\tau\ge \codeg
(P)$. 
As $\codeg(P)\ge\codeg_\rat(P)$,  we get $\codeg(P)=k+1$. 
Hence $k+1\ge \frac{n+3}2$, so that $k> \frac{n}2$.  

Since $L^k\cdot F=1$ for a general fiber of $\varphi$ and $\varphi$ is flat, $L^k\cdot Z=1$ for every fiber $Z$. Therefore all fibers are irreducible,
reduced, and of degree one in the corresponding embedding. It follows that for
every fiber $Z,$ $(Z,L|_Z)\cong (\pn{k},{\mathcal O}_{\pn{k}}(1)).$ Therefore 
$\varphi$ is a fiber bundle: $X=\pn{}_Y(\varphi_*L)$, where $\varphi_*L$ is a rank
$k+1$ vector bundle. Since $Y$ is toric, this bundle splits as a sum of line bundles 
$\varphi_*L=L_0 \oplus \ldots \oplus L_k$, and therefore $P$ is a strict Cayley polytope. 
Hence (1) implies (2).

Assume (2) holds. By Proposition \ref{localsplit} (with $s=1$), $P$ is $\rat$-normal, and $\codeg_\rat (P)=\tau=k+1$. Since $\codeg_\rat (P)$ is an integer, $\codeg(P)=\codeg_\rat (P)$, hence $\codeg(P)=k+1>\frac{n}2+1$, so $\codeg(P)\ge \frac{n+3}2$. Therefore (1) holds.

The equivalence of (2) and (3) is essentially contained in \cite{DiR06}; here is a brief sketch of the proof.

 Assume (2) holds. %We saw above that then $\codeg(P)=\codeg_\rat (P)=\tau$.
Since $k>\frac{n}2$, $(X,L)$ is defective with defect $\delta\ = 2k-n$ \cite[Prop.~5.12, p.~369]{LS86} hence $\delta=2\codeg(P)-2-n$, since $\codeg(P)=k+1$.
%(according to Ein, this was first observed by A. Landman and M. Reid, see \cite[p.~895]%{Ein2}). 
%For any nonsingular variety with positive dual defect $\delta$, the
%nef value is  $\tau = \frac{n+\delta}{2}+1$ (see  \cite[Cor.~6.6.8,
%p.~151]{BESO}). We get $\delta=2\tau-2-n=2\codeg (P)-2-n$, which shows that (3) holds.

Assume (3) holds, so that $P$ is defective with defect $\delta=2\codeg (P)-2-n\ge 1$. By \cite[Thm. 2]{DiR06}, then $P=\Cayley_\Sigma(P_0,\ldots,P_k)$ is a smooth strict Cayley polytope, with $k=\frac{n+\delta}2>\frac{n}2$ and $\codeg (P)=\frac{n+\delta}2+1=k+1$. 
\hfill $\square$
\medskip

We isolate the following result from the statement and proof of the previous
theorem. 

\begin{corollary} Let $P$ be a smooth $\rat$-normal lattice polytope  of
dimension $n$, and assume $\codeg(P) \ge \frac{n+3} 2$. Then, $\tau:=\tau(P)=
\codeg(P)$ is an integer and the associated polarized toric variety 
$(X,L)$ is defective with defect $\delta = 2 \tau -2 -n$.
\end{corollary}

The classification of {\em smooth  $\rat$-normal}  lattice polytopes of degree $0$ and $1$
follows from Theorem \ref{defective} (cf. \cite{Ba-Ni} for the general case).

\begin{corollary} Assume $P$ is a smooth, $\rat$-normal $n$-dimensional lattice polytope. Then
\begin{itemize}
\item[(1)] $\deg (P)=0$ if and only if $P=\Delta_n$,
\item[(2)] $\deg(P)=1$ if and only if $P={\rm Cayley}(P_0,\ldots,P_{n-1})$ is a
Lawrence prism (the $P_i$ are segments) or $P=2\Delta_2$.
\end{itemize}
\end{corollary}

\begin{proof} (1) Because  $\tau(P)> {\rm codeg}(P)-1=n\ge \frac{n+1}2,$ Lemma
\ref{line}Ê implies that  $\varphi_P$ contracts a line. It follows that 
$\tau(P)$ is an integer, $\tau(P)\ge n+1$ and thus 
$\tau(P)={\rm codeg}(P)=n+1$. Theorem \ref{defective} implies that $P$ is a
strict Cayley polytope with $k=n$. This is equivalent to  $P=\Delta_n$.

(2) For $n\ge 3$ the same argument as in (1) applies and gives that $P$ is
Cayley with $k=n-1$. 
Assume $n=2.$ Since $\codeg(P)=2$, $P\neq \Delta_{2}$ and $P$ has no interior
lattice points. Let $m$ be a vertex of $P$. Because $P$ is smooth, the first lattice
points on the two edges 
containing $m(\sigma)$ form a lattice basis. In this basis $m(\sigma)=(0,0)$,
and the lattice point $(1,1)$ is not an interior point of $P$. Therefore there
are only three possibilities: (i) One of the edges has length $>1$, the other
edge has length $1$, 
and the point $(1,1)$ is on a third edge parallel to $\rho_1$, which gives a
strict Cayley polytope. (ii) Both edges have length $1$, and $(1,1)$ is the
fourth vertex, hence 
$P=\Delta_{1}\times\Delta_{1}=\Cayley_{\Sigma(\Delta_1)}(\Delta_1,\Delta_1)$.
(iii) Both edges 
have length $2$, and there is only one other edge containing $(1,1)$. In this
case $P=2\Delta_2.$
\end{proof}

We close the paper with some loose ends. In the Introduction we 
conjectured
that any smooth lattice polytope $P$ with $\codeg(P) \ge
\frac{n+3}{2}$ is in fact $\rat$-normal. This conjecture is supported by   
\cite[7.1.8]{BESO}, which suggests 
that for $n \le 7$, if $P$ is a smooth lattice polytope and
$\codeg_\rat(P)>\frac{n}{2}$, then $P$ is $\rat$-normal. We also expect that
without any smoothness assumptions, all lattice 
polytopes $P$ with $\codeg(P) \ge \frac{n+3}2$ are indeed Cayley polytopes. 

\begin{ack} We thank the Institute for Mathematics and its Applications (IMA),
Minneapolis, where this project began. We also thank the Department of
Mathematics  of KTH and the G\"oran Gustafsson Foundation, as well as the Center of Mathematics for Applications (CMA) and
Department of Mathematics of the University of Oslo, for sponsoring several
visits, during which this paper was completed. We are grateful to Eva Maria 
Feichtner and Bernd Sturmfels for calling our attention to the paper
\cite{Ba-Ni}, and to Eduardo Cattani for useful discussions.  We would also like to thank Christian Haase and Benjamin Nill for very interesting discussions, and for showing us the alternative proof in Remark \ref{rmk:Haase}. Finally, we would like to thank the referee for thoughtful and most helpful comments.
\end{ack}


\begin{thebibliography}{00}

\bibitem%[BN07]
{Ba-Ni} Victor~V. Batyrev and Benjamin Nill.
\newblock Multiples of lattice polytopes without interior lattice points.
\newblock {\em Moscow Math. J.}, {\bf 7}(2):195--207, 2007.

\bibitem%[BS92]
{BESO92}Mauro~C. Beltrametti and Andrew~J. Sommese.
\newblock  On the adjunction theoretic classification of polarized varieties.
\newblock {\em J. reine angew. Math.} {\bf 427}:157--192, 1992.

\bibitem%[BS95]
{BESO} Mauro~C. Beltrametti and Andrew~J. Sommese.
\newblock {\em The adjunction theory of complex projective varieties},
 de Gruyter Expositions in Mathematics, Vol.~16.
\newblock Walter de Gruyter \& Co., Berlin, 1995.

\bibitem%[BSW92]
{BSW} Mauro~C. Beltrametti, Andrew~J. Sommese, and Jaros\l aw A.
Wi\'sniewski.
\newblock Results on varieties with many lines and their applications to
 adjunction theory.
 \newblock In \emph{Complex algebraic varieties (Bayreuth, 1990)}, 
 pp. 16--38. Lecture Notes in Math., 1507, Springer, Berlin,  1992.

\bibitem%[CaDR08]
{CaDR} Cinzia Casagrande and Sandra Di~Rocco.
\newblock Projective Q-factorial toric varieties covered by lines.
\newblock {\em Comm. Contemp. Math.}, {\bf 10}(3):363--389,
2008.

\bibitem%[CKM88]
{CKM88} Herbert Clemens, J\'anos Koll\`ar, and Shigefumi Mori.
\newblock Higher-dimensional  complex geometry. 
\newblock Ast\'erisque No. {\bf 166}, 1988.
%, 144 pp. 

\bibitem%[DR06]
{DiR06} Sandra Di~Rocco.
\newblock Projective duality of toric manifolds and defect polytopes.
\newblock {\em Proc. London Math. Soc. (3)}, {\bf 93}(1):85--104, 2006.

\bibitem%[DRS04]
{DRS04}Sandra Di Rocco and Andrew J. Sommese.
\newblock Chern numbers of ample vector bundles on toric surfaces.
\newblock  {\em Trans. Amer. Math. Soc.}, {\bf 356}(2):587--598, 2004.

%\bibitem%[Ein86]
%{Ein86} Lawrence Ein.
%\newblock Varieties with small dual varieties. I.
%\newblock \emph{Invent. Math.}, {\bf 86}(1):63--74, 1986.

%\bibitem%[Ein2]
%{Ein2} Lawrence Ein.
%\newblock Varieties with small dual varieties. II.
%\newblock \emph{Duke Math. J.}, {\bf 52}(4):895--907, 1985.

\bibitem%[Kaw84]
{kaw} Yujiro Kawamata.
\newblock The cone of curves of algebraic varieties.
\newblock {\em Ann. of Math. (2)}, {\bf 119}(3):603--633, 1984.

\bibitem%[HNP08]
{hnp} Christian Haase, Benjamin Nill, and Sam Payne.
\newblock Cayley decompositions of lattice polytopes and upper bounds for
$h^*$-polynomials.
\newblock Preprint, arXiv:0804.3667, 2008. (To appear in J. reine angew. Math.)

\bibitem%[LS86]
{LS86} Antonio Lanteri and Daniele Struppa.
\newblock Projective manifolds whose topology is strongly reflected in their hyperplane sections.
\newblock {\em Geom. Dedicata}, {\bf 21}(3):357--374, 1986.

\bibitem%[Mus02]
{Mustata} Mircea Musta{\c{t}}{\u{a}}.
\newblock Vanishing theorems on toric varieties.
\newblock {\em Tohoku Math. J. (2)}, {\bf 54}(3):451--470, 2002.

\bibitem%[Oda88]
{oda} Tadao Oda.
     \newblock {\em Convex bodies and algebraic geometry}.
\newblock Ergebnisse der Mathematik und ihrer Grenzgebiete (3), Vol. 15,
Springer-Verlag, Berlin 1988.
      
\bibitem%[Reid83]
{Reid83} Miles Reid. 
\newblock Decomposition of toric morphisms.
\newblock In \emph{Arithmetic and geometry}, Vol. II, pp.
 395--418, Progr. Math., 36, Birkh\"auser Boston, Boston, MA,  1983. 
 
 \bibitem%[Z95]
 {Z95}  G\"unter M. Ziegler.
\newblock \emph{Lectures on polytopes.}  
\newblock Graduate Texts in Mathematics, 152. Springer-Verlag, New York, 1995.

\end{thebibliography}
\end{document}